\documentclass[12pt]{amsart}

\usepackage[utf8]{inputenc}

\newtheorem{prop}{Proposition}[section]
\newtheorem{thm}{Theorem}[section]

\newtheorem{Def}{Definition}[section]
\newtheorem{lem}{Lemma}[section]

\newtheorem{remark}{Remark}
\newtheorem{cor}{Corollary}[section]

\usepackage{color,amscd}
\usepackage[draft]{changes}
\usepackage{enumerate}
\usepackage{mathtools}
\usepackage[abbrev]{amsrefs}
\usepackage{amssymb,amsmath,amsthm,xfrac,amsrefs}

\newcommand{\pa}{\partial}

\def\beq{\begin{eqnarray}}
\def\eeq{\end{eqnarray}}
\newcommand{\nn}{\nonumber}

\usepackage{graphicx}
\usepackage{amsmath, amsfonts,graphicx,subfigure,cite,comment,hyperref}
\usepackage{mathtools}
\usepackage[margin=1.5 in]{geometry}

\usepackage{graphicx}
\usepackage{mathtools}

\usepackage{amsmath,amsfonts,amsthm,amssymb,graphics,color,datetime}

\definechangesauthor[color=red]{nc}

\usepackage{tikz}
\usetikzlibrary{angles,quotes,calc}
\newcommand{\RN}[1]{%
  \textup{\uppercase\expandafter{\romannumeral#1}}%
}

\usepackage{tikz}
\usetikzlibrary{decorations.markings}
\usetikzlibrary{angles,quotes,calc}
\usetikzlibrary{calc}
 \usepackage{subfigure}

\def\beq{\begin{eqnarray}}
\def\eeq{\end{eqnarray}}

\newcommand{\Rea}{\operatorname{Re}} 

\numberwithin{equation}{section}

\newcommand{\Ima}{\operatorname{Im}}

        \definecolor{pink}{rgb}{1,0,1}
        \definecolor{purple}{rgb}{0.4,0.2,1}

\newcommand{\supp}{\operatorname{Supp}}

\newcommand{\Vol}{\operatorname{Vol}}

\newcommand{\eps}{\varepsilon}

\newcommand{\N}{\mathbb{N}}
\newcommand{\R}{\mathbb{R}}

\newcommand{\C}{\mathbb{C}}

\newcommand{\cC}{\mathcal{C}}
\newcommand{\cL}{\mathcal{L}}

\title{The Laplace spectrum on conformally compact manifolds}
\author{Nelia Charalambous} 
\author{Julie Rowlett}
\address{Nelia Charalambous,  Department of Mathematics and Statistics, University of Cyprus, 
P.O. Box 20537, 
CY-1678 Nicosia, 
Cyprus} 
\urladdr{\href{http://pythagoras.mas.ucy.ac.cy/nelia/index.html}{http://pythagoras.mas.ucy.ac.cy/nelia/index.html}}
\email{\href{mailto:charalambous.nelia@ucy.ac.cy}{charalambous.nelia@ucy.ac.cy}}

\address{Julie Rowlett,  Mathematical Sciences, Chalmers University of Technology and University of Gothenburg,  SE-412 96, Gothenburg, Sweden} 
\urladdr{\href{http://www.math.chalmers.se/~rowlett}{http://www.math.chalmers.se/~rowlett}}
\email{\href{mailto:julie.rowlett@chalmers.se}{julie.rowlett@chalmers.se}}

\date{June 2023}

\begin{document}

\maketitle

\begin{abstract} We consider the spectrum of the Laplace operator acting on $\cL^p$ over a conformally compact manifold for $1 \leq p \leq \infty$.  We prove that for $p \neq 2$ this spectrum always contains an open region of the complex plane.  We further show that the spectrum is contained within a certain parabolic region of the complex plane. These regions depend on the value of $p$, the dimension of the manifold, and the values of the sectional curvatures approaching the boundary.  
\end{abstract} 


\section{Introduction} \label{s:intro}
Conformally compact manifolds were introduced by Sir Roger Penrose in the 1960s \cites{penrose_63, penrose_65} as a tool to investigate asymptotic properties of radiative fields in general relativity.  A conformally compact manifold is a complete, smooth topological manifold with smooth boundary such that the Riemannian metric, known as a conformally compact metric, induces a unique conformal class of Riemannian metrics on the boundary (at infinity).  Examples include conformal compactifications of Einstein metrics.  As such, conformally compact manifolds provide a means to study the asymptotic behavior of solutions to the vacuum Einstein equations at null infinity.  They have been extensively researched by mathematicians and physicists.  An exhaustive list of citations is not feasible, so we simply mention an example from physics and from mathematics.  The physicists Stephan Hawking \& Don Page \cite{hawking_83}  used conformally compact manifolds to study thermodynamics of black holes in anti de Sitter space.  The mathematicians Charles Fefferman and Robin Graham 
used conformally compact manifolds to produce a host of scalar curvature invariants for conformal Riemannian manifolds \cite{feff_gram_85}. 


We now give a more detailed description of a conformally compact manifold, $M$.  It begins with an $n+1$ dimensional smooth topological manifold with smooth $n$ dimensional boundary, denoted by $\pa M = Y$.  The boundary, $Y$, is then a smooth, closed (compact), $n$ dimensional topological manifold. We do not assume that $Y$ is connected, but we do assume that there are finitely many boundary components.  We denote the closure of $M$ by $\overline M = M \cup \pa M$.  A \em boundary defining function \em is a smooth function $\rho: \overline{M} \to [0, \infty)$ such that $\rho^{-1} \{ 0 \} = \pa M$.  Moreover, $\rho$ is required to vanish \em simply \em at the boundary, so that $d\rho \neq 0$ on $\pa M$.  As in \cite{mazzeo_88}, we say that $(M,g)$ is \em conformally compact \em if the Riemannian metric satisfies
\[ g = \frac{\overline{g}}{\rho^2}.\]
Above, $\overline{g}$ is a Riemannian metric on $M$ which extends smoothly to a smooth Riemannian metric on $Y$.  Expressed in this form, the sectional curvature of the conformally compact manifold tends to $-|d\rho(p)|^2 = - \frac{1}{|\partial \rho(p)|^2}$ at a boundary point $p$. Since $\rho$ is smooth, the sectional curvature tends therefore to strictly negative but not necessarily identical values at the boundary. 

One of the most fundamental partial differential operators on a Riemannian manifold is the Laplace operator.  On a manifold of dimension $n+1$ with Riemannian metric $g$, the Riemannian metric  at each point can be expressed in local coordinates  as an $(n+1) \times (n+1)$ positive definite symmetric matrix, denoted by $(g_{ij})$.  The determinant of this matrix is denoted $\det(g)$ and its inverse at each point is denoted $(g^{ij})$.  The Laplace operator is then 
\[ \Delta = - \frac{1}{\sqrt{\det(g)}} \sum_{i,j=1} ^{n+1} \pa_i g^{ij} \sqrt{\det(g)} \pa_j.\]
 
A natural question for geometric operators such as the Laplacian, is to what extent do the geometric and topological properties of the manifold influence the analytical features  of the operator such as its spectrum.  To make this more precise, we recall the general definitions of spectrum and resolvent set for an operator $H$ on a Banach space $B$.  A point $\lambda \in \C$ belongs to the resolvent set of $H$ if $H-\lambda I$ has a bounded inverse on $B$, with $I$ the identity operator.  The complement of the resolvent set in $\C$ is the spectrum of the operator.

On a non-compact manifold, such as a conformally compact manifold, $M$, the Laplace operator is usually defined for smooth, compactly supported functions, $\cC^\infty _c (M)$.  It is then extended to a Hilbert space contained in $\cL^2(M)$.  The minimal domain is the closure of the graph of $\Delta$ on $\cC^\infty _c(M)$ with respect to the $\cL^2$ norm.  The maximal domain consists of all $f \in \cL^2$ such that the distribution $\Delta f$ can be represented by an $\cL^2$ function.  Since conformally compact manifolds are complete, the minimal and maximal domains are equal, so $\Delta$ is essentially self-adjoint \cites{roelcke_60,strichartz_83}.  

The spectrum of the Laplace operator on a conformally compact manifold $M$ acting on its canonical domain in $\cL^2(M)$ was first considered by Mazzeo \cite{mazzeo_88}.  The essential spectrum is $[n^2 \alpha_0^2/4, \infty)$, with $\alpha_0^2$ the minimum of $|d\rho(p)|^2$ over the boundary, for a boundary defining function $\rho$.  The essential spectrum is absolutely continuous, and there are no embedded eigenvalues except possibly at $n^2 \alpha_0^2/4$. There could also exist a finite set of eigenvalues of finite multiplicity in the interval $(0, n^2 \alpha_0^2/4)$.  We note that by the definition of a conformally compact manifold, $|d\rho|$ is independent of the choice of boundary defining function, and it is strictly positive and continuous on the boundary which is compact.  Hence $\alpha_0>0$.  The variability of the sectional curvatures at the boundary poses challenges to the analysis.  According to Mazzeo, ``the multiplicity of the continuous spectrum presumably varies in a complicated manner in $[n^2 \alpha_0 ^2/4, n^2 \alpha_1 ^2 /4]$, with $\alpha_1^2$ the maximum of $|d\rho|^2$ over the boundary.  Indeed, when the sectional curvatures tend to a negative constant at the boundary, Mazzeo and Melrose \cite{mm_87} had previously obtained the meromorphic continuation of the resolvent operator on such spaces, known as asymptotically hyperbolic manifolds.  Similarly, in \cite{emm_91} Epstein, Melrose, and Mendoza obtained such results for the resolvent of the Laplacian on strictly pseudoconvex domains.  Ten years later, in 2001, Borthwick developed the scattering theory of conformally compact manifolds. Following Mazzeo \cite{mazzeo_88}, Borthwick used the spectral parameter $s$ such that the relation to the eigenvalue $\lambda$ is 
\[ \lambda = \alpha_0 ^2 s (n-s).\]
Borthwick then implemented a microlocal construction to obtain the meromorphic continuation of the Schwartz kernel of the resolvent operator $(\Delta - \lambda I)^{-1}$ to the complex plane (for the parameter $s \in \C$), minus a certain collection of intervals.  The portion of the continuous spectrum inaccessible by meromorphic continuation indeed corresponds to the range previously mentioned by Mazzeo, $\lambda \in [n^2 \alpha_0 ^2/4, n^2 \alpha_1 ^2 /4]$ and is the result of the variability of the curvature at the boundary.  Related work include the analytic continuation of the resolvent on symmetric spaces of noncompact type \cite{mazzeo_vasy_05}.  S\'a Barreto and Wang \cite{sabarreto_wang_16} proved that if a conformally compact manifold is non-trapping, then the semiclassical scattering matrix is a semiclassical Fourier integral operator which quantizes the scattering relation. They also proved resolvent estimates and showed that there is a resonance-free region near the continuous spectrum.  The non-trapping assumption, that there are no closed geodesics on the manifold, is often assumed in order to obtain resolvent estimates.  

Here, we investigate the spectrum of the Laplace operator on a conformally compact manifold acting on the Banach spaces $\cL^p(M)$ for $1 \leq p \leq \infty$.  We allow the curvatures to be variable at the boundary, and we do not make any non-trapping assumptions.  Consequently, it is not immediately apparent that the methods of the aforementioned works would yield the results obtained here.  Since conformally compact manifolds are complete, there is a canonical domain for the Laplace operator acting on $\cL^p(M)$ for $p$ in this range. For the Laplacian on functions, using Davies' Theorems 1.3.2, 1.3.3, 1.4.1 in~\cite{davies89}, one can obtain that the heat operator $e^{-t\Delta}$ of the Laplacian operator on $\cL^2(M)\cap \cL^p(M)$ can be extended to a contraction semigroup on $\cL^p(M)$ for all $1 \leq p < \infty$. For $p=\infty$, the heat operator on $\cL^{\infty}(M)$ is defined as the dual of the operator on $\cL^1(M)$. This allows us to define the infinitesimal generator  of the semigroup $e^{-t \Delta}$ on $\cL^p(M)$ for all $1\leq p\leq\infty$. We refer to this generator as the Laplacian on $\cL^p(M)$. Note that $\cC^\infty _c (M)$ is also a core for this operator.

There is significant motivation to study the Laplace operator acting on these Banach spaces.  For example, the natural space to study heat diffusion is $\cL^1(M)$ because if a function $u(t, x) \geq 0$ is the heat distribution at time $t$, the total amount of heat in any region is given by the $\cL^1$ norm of $u$ over that region.  Consequently, the $\cL^1$ norms of solutions to the heat equation have a physical meaning.  This is sometimes known as the \em heat content. \em  Although this provides motivation for the study of the Laplace operator on $\cL^1$, this space is more difficult to handle than the Hilbert space $\cL^2$ or the reflexive $\cL^p$ spaces for $p>1$.  In particular, the heat semigroup is bounded analytic on $\cL^p$ for $p>1$, but this is not in general true for the heat semigroup on $\cL^1$.  

There are numerous results in the literature for $\cL^p$ spectral theory for differential operators on domains of euclidean space including but not limited to \cite{arendt94, davies95,davies89,davies97,kunstmann99}. Hempel and Voigt proved that the $\cL^p$ spectrum of the Laplacian and and a class of Schr\"odinger operators over euclidean space  is independent of $p$ for all $1\leq p\ \leq \infty$ \cite{HV1}. The $p$-indendence result can be generalized to various classes of elliptic and Schr\"odinger-type operators over manifolds, metric meausure spaces, as well as operators acting on other bundles by controlling the potential term of the operator, as well as assuming some form of subexponential growth for the volume of the manifold \cite{ChGr,ChLu2,Kordu,sturm93}.  In some cases knowledge of the $\cL^p$ spectra can be used to obtain information on the decay of $\cL^2$ eigenfunctions of the Laplacian as done in \cite{taylor89}. For example, to calculate the $\cL^2$ spectrum of the Laplacian on complete Riemannian manifolds with non-negative Ricci curvature, in \cite{wang97} J. Wang first calculated the $\cL^1$ spectrum and then used a result of Sturm \cite{sturm93} which showed that the $\cL^p$ spectrum does not depend on $p$.  

On the other hand there are various results which show that the $\cL^p$ spectrum of the Laplace-Beltrami operator on a Riemannian manifold may depend non-trivially on $p$ and is related to the volume growth of the manifold. For example, Davies, Simon and Taylor proved in \cite{DST} that the  $\cL^p$ spectrum of the Laplace-Beltrami operator on hyperbolic space does depend on $p$ and is a parabolic region of the complex place which collapses to a closed subset of the real line for $p=2$. In the same article they generalized their result to noncompact geometrically finite quotients of hyperbolic space, as well as quotients with finite volume or which are cusp-free. Taylor showed that the $\cL^p$ spectrum of the Laplacian is also a parabolic region whenever the manifold is a symmetric space of noncompact type \cite{taylor89}.  Ji and Weber considered the case of locally symmetric spaces of any rank providing further details about the nature of the spectrum and the   $\cL^p$ eigenvalues \cite{JW1,JW}.  In \cite{taylor89} Taylor also proved that the $\cL^p$ spectrum of certain functions of the Laplace operator (including the Laplacian) is contained in a parabolic region whenever the underlying space is a noncompact manifold with bounded geometry, injectivity radius uniformly bounded below,  and at most exponential volume growth \cite{taylor89}.

We see from the aforementioned results that the set of manifolds over which the $\cL^p$ spectrum of the Laplace operator on functions is $p$ dependent has been restricted thus far to quotients of hyperbolic space. In this article, we will see that conformally compact manifolds provide an ideal setting of much more general spaces where this is also true.  Our results show that the spectrum depends on $p$, the dimension of the manifold, and its sectional curvature near the conformal boundary. 

\begin{thm} \label{th1} 
Let $M$ be a conformally compact manifold of dimension $n+1$ with boundary defining function $\rho$.  Let  $\alpha_0^2$ and $\alpha_1^2$ be, respectively, the minimum and maximum of $|d\rho|^2$ on the boundary of $M$. Then, the spectrum of the Laplace operator acting on $\cL^1(M)$ contains the region 
\[ \left\{ x + i y : x \geq \frac{y^2}{n^2 \alpha_1^2}, \, y \in \R \right\} \subset \C. \] 
Moreover, the  $\cL^1$  spectrum is contained within the region 
\[ \left\{ x + i y : x \geq \frac{y^2}{n^2 \alpha_1^2} - \frac{n^2 \alpha_1^2}{4} +\lambda_1, \, y \in \R \right\}.\] 
Above, $\lambda_1$ is the bottom of the $\cL^2(M)$ spectrum.  This is either an isolated eigenvalue of finite multiplicity contained in $(0, n^2 \alpha_0^2/4)$, or if the $\cL^2(M)$ spectrum contains no isolated eigenvalues of finite multiplicity, then $\lambda_1=\frac{n^2  \alpha_0^2}{4}$.
\end{thm} 


If the manifold is not only conformally compact, but also asymptotically hyperbolic, and there are no isolated eigenvalues in the $\cL^2$ spectrum, then $\alpha_0=\alpha_1$, and we obtain the $\cL^1$ spectrum precisely. 

\begin{cor} \label{cor1}
Under the same hypotheses as Theorem \ref{th1}, assume further that $M$ is asymptotically hyperbolic in the sense that $|d\rho|^2$ is a positive constant on $\pa M$, and the $\cL^2(M)$ spectrum contains no isolated eigenvalues of finite multiplicity.  Then the Laplace operator acting on $\cL^1 (M)$ has spectrum equal to 
\[ \left\{ x + i y : x \geq \frac{y^2}{n^2 \alpha_1^2}, \, y \in \R \right\} \subset \C. \] 
Above $\alpha_1^2 = |d\rho|^2$ on the boundary of $M$. 
\end{cor} 

A further consequence of our study of the $\cL^1$ spectrum is the exponential rate of volume growth of any conformally compact manifold.  For the precise definition, we refer to Definition \ref{def:volgrow} in \S \ref{s:resolvent}.  Heuristically, an exponential rate of volume growth equal to $\kappa$ means that the volume of a ball of radius $r$ is of the order $e^{\kappa r}$ as $r \to \infty$.   
\begin{cor} \label{cor2}
Under the same hypotheses as Theorem \ref{th1}, the exponential rate of volume growth of an $n+1$ dimensional conformally compact manifold is $n \alpha_1$. 
\end{cor}

Our next result shows that a one-parameter family of parabolic regions are contained in the $\cL^p$ spectrum for  $1 \leq p \leq 2$.  We further determine a parabolic region that contains the $\cL^p$ spectrum. Note that for $q=\frac{p}{p-1} \geq 2$, the Laplacian on $\cL^q$ is the dual operator to the Laplacian on $\cL^p$ \cite[Theorem 1.4.1]{davies89}.  As a result the spectrum of the Laplace operator acting on $\cL^q(M)$ coincides (by duality) with the spectrum of Laplace operator acting on $\cL^p(M)$. This is also true for the Laplacian acting on $\cL^{\infty}(M)$ by its definition. Hence, our results for  $1 \leq  p \leq 2$ provide a comprehensive description for the $\cL^p$ spectrum of the Laplacian for all $1\leq p\leq \infty$.   

\begin{thm} \label{th2} 
Assume that $1 \leq p \leq2$.  The spectrum of the Laplace operator acting on $\cL^p(M)$ for a conformally compact manifold $M$ of dimension $n+1$ contains the region 
\[ \bigcup_{A \in |d\rho|^2(Y)} \left\{x+iy \in \C: x\geq \frac{y^2}{A n^2 (1-2/p)^2} + A \frac{n^2}{p} \left( 1 - \frac 1 p \right), \, y \in \R \right\}. \]
Above, $|d\rho|^2(Y)$ is the image of $|d\rho|^2$ over the boundary of $M$.  Note that for $p=2$ the region reduces to an interval in the real line.

Moreover,  the  $\cL^p(M)$  spectrum is contained within the parabolic region 
\[ \left \{ x + i y : x \geq \lambda_1 - \frac{ \left( \frac 2 p - 1 \right)^2 n^2 \alpha_1^2}{4}+ \frac{y^2}{n^2 \alpha_1^2 \left( \frac 2 p - 1\right)^2}  \, \, , y \in \R \right\} \]  
Above, $\lambda_1$ is the bottom of the $\cL^2(M)$ spectrum.  This is either an isolated eigenvalue of finite multiplicity contained in $(0, n^2 \alpha_0^2/4)$, or if the $\cL^2(M)$ spectrum contains no isolated eigenvalues of finite multiplicity, then $\lambda_1=\frac{n^2  \alpha_0^2}{4}$. Above $\alpha_0^2$ and $\alpha_1^2$ are respectively the minimum and maximum values of $|d\rho|^2$ over the boundary.  
The respective results for $\frac 1q=1-\tfrac{1}{p}$ with $q \geq 2$,  follow by duality. 
\end{thm} 

We note that it is straightforward to show that the regions given in Theorems \ref{th1} and \ref{th2} are invariant under $p \mapsto \frac{p}{p-1}$.  There may exist isolated points in the spectrum in which case  they  must be contained within the larger parabolic regions given in Theorem \ref{th1} and \ref{th2}.  In the following result we further characterize the $\cL^p$ eigenvalues for $1 \leq p < 2$,   showing that any  $\cL^p$ eigenvalue must also be an $\cL^2$ eigenvalue.

\begin{thm} \label{th:discrete}
Let $M$ be a conformally compact manifold. If $\lambda_i$ is an $\cL^2$-eigenvalue for the Laplacian, then it must also be an $\cL^p$-eigenvalue for all $p>2$. Moreover,  if $\lambda_i$ is an $\cL^p$-eigenvalue for the Laplacian for some $1\leq p<2$, then it must also be an $\cL^2$-eigenvalue. In particular, if $\lambda$ is a point in the $\cL^p$ spectrum for $1\leq p <2$ which is not in the $\cL^2$ spectrum, then it cannot be an eigenvalue.
\end{thm}
The above result shows that any point in the $\cL^p$ spectrum for $1 \leq p < 2$ that does not belong  to the $\cL^2$ spectrum cannot be an $\cL^p$ eigenvalue. In particular, any point in the $\cL^p$ outside the real line must belong to the $\cL^p$ essential spectrum for $1 \leq p < 2$.

\subsection{Organization and outline} \label{ss:oo} 
In \S \ref{s:construction} we construct families of approximate eigenfunctions.  These are used to prove that the parabolic region in Theorem \ref{th1} is contained within the $\cL^1$ spectrum, and the one-parameter family of parabolic regions in Theorem \ref{th2} are contained in the $\cL^p$ spectrum for $1\leq p \leq 2$.  These regions are depicted in Figures \ref{fig:L1spec} and \ref{fig:Lpspec}, respectively. We also provide the enveloping curve for the family of parabolas of Theorem \ref{th2} and depict this in Figure \ref{fig:envelope}.In \S \ref{s:resolvent} we demonstrate an upper bound for the exponential rate of volume growth of a conformally compact manifold and use this to show that the exterior of certain parabolic regions of the complex plane must be contained in the resolvent set.  Using the parabolic region contained in the $\cL^1$ spectrum, we obtain a lower bound for the exponential rate of volume growth that is equal to the upper bound, thereby specifying the exponential rate of volume growth precisely.  We use this to 
obtain the parabolic regions in Theorems \ref{th1} and \ref{th2} that contain the $\cL^1$ and $\cL^p$ spectra, respectively.  In Figures \ref{fig:L1spec_res} and \ref{fig:Lpspec_res} we show examples of the parabolic regions containing and contained in the $\cL^1$ and $\cL^p$ spectra, respectively.  We prove Theorem \ref{th:discrete} in \S \ref{s:discrete} and offer concluding remarks in \S \ref{s:conclude}. 

\section{Construction of approximate eigenfunctions} \label{s:construction} In this section we construct approximate eigenfunctions in order to demonstrate that certain regions of the complex plane are contained in the spectrum of the Laplace operator acting on $\cL^p$ for $1 \leq p \leq 2$.  Interestingly, the regions depend on $p$, the dimension $n+1$, and the geometry at infinity (at the boundary).  We obtain that the $\cL^1$ spectrum contains a parabolic region as depicted in Figure \ref{fig:L1spec}.  For $p \in (1,2)$ we obtain that the $\cL^p$ spectrum contains the closures of the interiors of a one-parameter family of parabolas as depicted in Figure \ref{fig:Lpspec}.  For $p=2$, our results show that the ray $[\alpha_0^2n^2/4, \infty)$ is contained in the $\cL^2$ spectrum as previously shown by Borthwick \cite{borthwick}.

By \cite[Proposition 3.1]{borthwick}, there exists a boundary defining function $x$ and a compact set $K$ such that the metric 
\begin{equation} \label{eq:special_g}  g = \frac{dx^2}{\alpha(y)^2 x^2} + \frac{h(x,y,dy)}{x^2} + O(x^\infty), \quad \textrm{ on } M \setminus K \end{equation}
with 
\begin{equation} \label{eq:product_x1} M \setminus K \cong (0, x_1) \times Y, \textrm{ with local coordinates } (x, y) = (x, y_1, \ldots, y_n), \end{equation} 
for some fixed $x_1 > 0$.  Above, the coordinate $x$ is a boundary defining function, and $y=(y_1, \ldots, y_n)$ are local coordinates on the boundary $Y$.  In \eqref{eq:special_g} the function $\alpha(y)$ is smooth and without loss of generality may be chosen to be strictly positive.  In \eqref{eq:special_g}, $h(x,y,dy)$ is a smooth family of metrics on $Y$ that depends smoothly on $x$ as a parameter and converges uniformly to a fixed smooth metric $h(0,y,dy)$ on $Y$ as $x \to 0$.  We use the following notations:  
\begin{enumerate} 
\item A function $f$ defined on $M \setminus K$ is $O(x^\infty)$ if for any $n \in \N$ there exists a constant $C_n >0$ such that 
\[ |f(x,y)| \leq C_n x^n \textrm{ on } M \setminus K.\]  
\item Two functions $\varphi$ and $\psi$ defined on $M \setminus K$ satisfy $|\varphi(x,y)| \lesssim |\psi(x,y)|$ if there is a constant $c>0$ such that 
\[ |\varphi(x,y)| \leq c |\psi(x,y)| \textrm{ on } M \setminus K. \]  
\item Two functions $\varphi$ and $\psi$ defined on $M \setminus K$ satisfy $|\varphi(x,y)| \approx |\psi(x,y)|$ if 
\[|\varphi(x,y)| \lesssim |\psi(x,y)| \lesssim |\varphi(x,y)|. \] 
\end{enumerate} 

Expressed as in \eqref{eq:special_g}, for a point $y$ in the boundary, the sectional curvature tends to $-\alpha(y)^2$ as $x \to 0$ with respect to the local coordinates $(x,y)$ that are valid in a neighborhood of the boundary point $y$.  For such a metric \eqref{eq:special_g}, the Laplace operator takes the form 
\begin{multline} \label{eq:Delta} \Delta = \alpha^2 \left[ - (x \pa_x)^2 + n x \pa_x - x^2 (\pa_x \log \sqrt h) \pa_x \right] \\ + x^2 \Delta_h - x^2 (\pa_i \log \alpha) h^{ij} \pa_j + O(x^\infty). \end{multline}  
Above, $\Delta_h$ is the Laplace operator on $Y$ with respect to the metric $h(x,y, dy)$, and $\sqrt h$ is abbreviated notation for $\sqrt {\det h}$.  We note that 
\[ \Delta_h = -\frac{1}{\sqrt{h}} \sum_{i,j=1} ^n \pa_i h^{ij} \sqrt{h} \pa_j.\] 

We will build a family of approximate eigenfunctions that are compactly supported in $M\setminus K \cong (0, x_1) \times Y$.  To motivate our construction, we consider a function of the form $x^\lambda$ for $\lambda \in \C$.  We apply the Laplace operator \eqref{eq:Delta} to such a function and obtain 
\begin{multline} \label{eq:approx_ef0} \Delta x^\lambda = \alpha^2 \left[ - \lambda^2 x^\lambda + n \lambda x^\lambda - \lambda x^{\lambda+1} \pa_x (\log \sqrt h) \right] + x^2 \Delta_h (x^\lambda) + O(x^\infty)\\ 
 = \alpha^2 \lambda(n-\lambda) x^\lambda - \alpha^2 \lambda x^{\lambda+1} \pa_x (\log \sqrt h) + O(x^\infty). \end{multline}
If $\alpha$ were to be a constant, and if $x^\lambda$ were to be in $\cL^p$, then as $x \to 0$, the function $x^\lambda$ would be, up to the error term $O(x^{\lambda+1})$, an eigenfunction with eigenvalue $\alpha^2 \lambda (n-\lambda)$.  This consideration inspires our construction of approximate eigenfunctions.

The approximate eigenfunctions will be of the form 
\beq F(x,y) = \phi(x) b(y) x^\lambda. \label{eq:F(x,y)} \eeq 
Above, $\phi$ and $b$ are smooth cut-off functions such that $F$ is supported in $M \setminus K$.  We calculate for such a function $F(x,y)$
\beq  \Delta F(x,y) &=& \alpha^2(y) \lambda (n-\lambda) F(x,y) \nn \\ 
 &+& x^{\lambda + 1} \alpha^2(y) \left( \phi'(x) b(y) (n-2\lambda - 1) - \lambda \phi(x) b(y) \partial_x \log \sqrt {h(x,y)} \right) \nn \\  &+& x^{\lambda+2} \left( - \alpha^2(y) \phi''(x) b(y) - \alpha^2(y) \phi'(x) b(y) \partial_x \log \sqrt{h(x,y)} \right) \nn \\ &+& x^{\lambda+2} \left( \phi(x) \Delta_h b(y) - \phi(x) \sum_{i,j=1} ^n h^{ij} \partial_i (\log \alpha(y)) \partial_j (b(y)) \right) \nn \\ &+& O(x^\infty). \label{eq:laplace3} \eeq 
We will use the first term on the right side of the first line in \eqref{eq:laplace3} to approximate an eigenfunction.  The function 

\[ \alpha^2:Y \to |d\rho|^2(Y)  \subset (0, \infty), \]  
where $|d\rho|^2(Y)$ is the image of $|d\rho|^2$ over the boundary of $M$.   Since $\alpha^2$ is smooth, and $Y$ is a smooth $n$-dimensional manifold, for any value $A \in |d\rho|^2(Y)$, and for any $\epsilon > 0$, the inverse image $(\alpha^2)^{-1}(A-\epsilon, A+\epsilon)$ is an open set in $Y$.  It therefore contains a ball of radius $r=r(\epsilon, A)>0$.  We denote this ball by $B_\epsilon$.  On $B_\epsilon$ we have the estimate 
\[ |\alpha^2(y) - A| < \epsilon.\]
We will use this to construct a linearly independent family $\{ F_\epsilon \}_{\epsilon > 0}$ of approximate eigenfunctions that satisfy 
\begin{equation} \label{eq:lp_eval_est} ||\Delta F_\epsilon - \Lambda F_\epsilon ||_p \leq \epsilon ||F_\epsilon||_p + O(x^\infty), \end{equation} 
with eigenvalue parameter 
\[ \Lambda := A \lambda(n-\lambda) = \Lambda(n, \lambda, A).\]

We choose the function $b(y)$ to be smooth, $b:Y \to [0, 1]$, the support of $b$ is contained in $B_\epsilon$, and 
\begin{equation} \label{eq:bpp} ||b||_p ^p := \int_Y |b(y)|^p d\Vol_h(y) > 0 \quad \forall 1 \leq p \leq 2. \end{equation}
Above, the measure of integration $d\Vol_h(y)$ is induced by the Riemannian metric, $h(x=0, y, dy)$ on $Y$.  We may use the slight abuse of notation for simplicity and write $dy$ rather than the more cumbersome $d\Vol_h (y)$.

As we have defined it, the function $b$ implicitly depends on $\epsilon$.  Next, we consider the cutoff function $\phi(x)$.  As indicated by the notation, this function depends only on $x$ and is independent of $y$.  We choose this function to have compact support contained in $(\delta, x_1)$, for $\delta>0$ that shall tend to zero.  We therefore assume without loss of generality that 
\[ 2 \delta < \min \left\{ \frac{x_1}{2}, 1\right\}, \textrm{ and } \phi(x) = 1 \textrm{ for } x \in \left[ 2 \delta, \frac{x_1}{2} \right].\]
The cutoff function $\phi$ therefore depends on the parameter $\delta > 0$.
For the sake of readability, we simply write $F$ rather than $F_\epsilon$ and estimate the $\cL^p$ norm of such an $F(x,y)$ from below,  
\beq ||F||_p ^p &=& \int_{\supp(\phi)} \int_{\supp(b)} |x^{\lambda p}| \, |\phi(x)|^p |b(y)|^p x^{-n-1} \sqrt{h} \alpha^{-1} dy dx +O(x^\infty) \nn \\
&\geq & \int_{2 \delta} ^{x_1/2} \int_{\supp(b)} |x^{\lambda p - n - 1}| \, |b(y)|^p \sqrt{h(x,y)} \alpha(y)^{-1} dy dx + O(x^\infty) \nn \\ 
&\gtrsim&   \left( \int_{2 \delta} ^{x_1/2} |x^{\lambda p -n-1}| dx \right) ||b||_p ^p. \label{eq:Fpp_below} 
\eeq
This follows because $\sqrt{h(x,y)}$ and $\alpha^{-1} (y)$ are uniformly bounded from below (and above) on $[2\delta, x_1/2] \times Y$.  We similarly estimate from above 
\beq ||F||_p ^p &=& \int_{\supp(\phi)} \int_{\supp(b)} |x^{\lambda p}| \, |\phi(x)|^p |b(y)|^p x^{-n-1} \sqrt{h} \alpha^{-1} dy dx +O(x^\infty) \nn \\
&\leq & \int_{\delta} ^{x_1} \int_{\supp(b)} |x^{\lambda p - n - 1}| \, |b(y)|^p \sqrt{h(x,y)} \alpha(y)^{-1} dy dx + O(x^\infty) \nn \\ 
&\lesssim&   \left( \int_{\delta} ^{x_1} |x^{\lambda p -n-1}| dx \right) ||b||_p ^p. \label{eq:Fpp_above} 
\eeq
Next we calculate 
\beq \int_{2 \delta} ^{x_1/2} |x^{\lambda p - n - 1}| dx &=& \begin{cases} \displaystyle{\frac{ (x_1/2)^{\Rea(\lambda)p - n} - (2\delta)^{\Rea(\lambda)p-n}}{\Rea(\lambda)p - n},} & \Rea(\lambda) \neq \frac n p, \\ 
& \\
\displaystyle{\log \left(  x_1/(2 \delta) \right),} & \Rea(\lambda) = \frac n p,  \end{cases}  \nn \\  \\
 \int_{\delta} ^{x_1} |x^{\lambda p - n - 1}| dx &=& \begin{cases} \displaystyle{ \frac{ x_1 ^{\Rea(\lambda)p - n} - \delta^{\Rea(\lambda)p-n}}{\Rea(\lambda)p - n},} & \Rea(\lambda) \neq \frac n p, \\ 
 &\\
 \displaystyle{\log \left( x_1 / \delta \right),} & \Rea(\lambda) = \frac n p, \end{cases}. \nn \eeq
 To control the terms in \eqref{eq:laplace3}, we will need $||F||_p$ to be large.  We will therefore henceforth assume that $\lambda \in \C$ satisfies 
 \beq \Rea(\lambda) \leq \frac n p. \nn \eeq 
We then obtain that 
\[\int_{2 \delta} ^{x_1/2} |x^{\lambda p - n - 1}| dx \approx \int_{\delta} ^{x_1} |x^{\lambda p - n - 1}| dx \approx \begin{cases} \delta^{\Rea(\lambda)p -n}, & \Rea(\lambda) < \frac n p, \\ \log(1/\delta), & \Rea(\lambda) = \frac n p. \end{cases} \]
Consequently, we estimate 
\beq  ||F||_p ^p \approx  ||b||_p ^p \times \begin{cases} \delta^{\Rea(\lambda)p -n}, & \Rea(\lambda) < \frac n p, \\ \log(1/\delta), & \Rea(\lambda) = \frac n p. \end{cases} \label{eq:Fpp} \eeq

We calculate using \eqref{eq:laplace3} and the triangle inequality 
\beq \left \Vert \Delta F(x,y) - \Lambda F(x,y)\right \Vert_p & \leq &  \left \Vert \left( \alpha^2(y) \lambda (n-\lambda) - \Lambda\right) F(x,y) \right \Vert_p \nn \\ 
&+& \left \Vert x^{\lambda + 1} \alpha^2(y) \phi'(x) b(y) (n-2\lambda - 1) \right \Vert_p \nn \\
&+& \left \Vert x^{\lambda+1}  \alpha^2(y)  \lambda \phi(x) b(y) \partial_x \log \sqrt {h(x,y)}  \right \Vert_p \nn \\
&+& \left \Vert x^{\lambda+2} \alpha^2(y) \phi''(x) b(y) \right \Vert_p \nn \\ 
&+& \left \Vert x^{\lambda+2} \alpha^2(y) \phi'(x) b(y) \partial_x \log \sqrt{h(x,y)} \right \Vert_p \nn \\  
&+& \left \Vert x^{\lambda+2} \phi(x) \Delta_h b(y) \right \Vert_p \nn \\ 
&+& \left \Vert x^{\lambda+2} \phi(x) \sum_{i,j=1} ^n h^{ij} \partial_i (\log \alpha(y)) \partial_j (b(y)) \right \Vert _p \nn \\ 
&+& O(x^\infty) \nn \\ 
&=& I+II+III+IV+V+VI+VII+O(x^\infty). \label{eq:Delta_Fp_est1} 
\eeq 

On the support of $b$, we have the estimate 
\[ |\alpha^2(y) \lambda (n-\lambda) - \Lambda| < \epsilon \implies I^p \leq \epsilon^p ||F(x,y)||^p _p. \]
We estimate next 
\[ |II|^p \leq |n-2\lambda -1|^p \int |x^{\lambda p + p-n-1}| \, |\phi'(x)|^p |\alpha(y)^{2p-1}| |b(y)|^p \sqrt h dx dy \]
\[ \lesssim ||b||_p ^p \int_{\delta} ^{x_1}|x^{\lambda p + p-n-1}| \, |\phi'(x)|^p dx. \]
The support of $\phi'$ is contained in $[\delta, 2 \delta] \cup [x_1/2, x_1]$.  The parameter $\delta$ will tend to zero, while $x_1>0$ is fixed.  Consequently, we may estimate $|\phi'|^p$ to be on the order of $\delta^{-p}$ on $x \in [\delta, 2\delta]$, and $|\phi'|^p$ to be uniformly bounded on $x \in [x_1/2, x_1]$.  We then estimate 
\beq |II|^p &\lesssim& ||b||_p ^p \int_{\delta} ^{2\delta} \delta^{-p} |x^{\lambda p +p-n -1}| dx \nn \\ 
&\approx& ||b||_p ^p \times \begin{cases} \delta^{\Rea(\lambda)p -n}, & \Rea(\lambda) < \frac n p, \\ \log(2), & \Rea(\lambda) = \frac n p. \end{cases}  \label{eq:n_on_p} \eeq
We need this to be small in comparison to $||F||_p ^p$.  For this reason we choose 
\begin{equation} \label{eq:relambda2} \Rea(\lambda) = \frac n p \end{equation} 
and provide the remaining estimates with this choice. Then we obtain that 
\[ |II|^p \lesssim ||b||_p ^p.\]
By the assumptions on $b$, $|II|$ is uniformly bounded from above as both $\epsilon, \delta \to 0$. 

Next we consider 
\beq |III|^p &\leq& |\lambda|^p \int |x^{\lambda p +p-n-1}| |\alpha(y)|^{ 2 p -1} \,|\phi(x)|^p |b(y)|^p |\partial_x \log \sqrt h|^p \sqrt h dx dy \nn \\
& \lesssim & ||b||^p _p \int_{\delta} ^{x_1} |x^{\lambda p + p - n - 1} | dx \nn \\ 
&\approx & ||b||_p ^p. \nn \eeq
since $p \geq 1$.  Next we estimate 
\beq |IV|^p &=& \int |x^{(\lambda+2)p}| \alpha^{2p -1} |\phi''|^p |b|^p x^{-n-1} \sqrt h dx dy \nn \\ 
&\lesssim& ||b||_p ^p \int_\delta ^{x_1} x^{2p-1} |\phi''|^p dx \nn  \; \lesssim  \; ||b||_p ^p \int_{\delta} ^{2\delta} x^{2p-1} |\phi''|^p dx \nn \\ 
& \approx & ||b||_p ^p.  \nn 
\eeq 
Above, we used the estimate of $|\phi''| \lesssim \delta^{-2}$ for $x \in [\delta, 2 \delta]$. Next, since  $p \geq 1$ we also obtain 
\beq |V|^p &=& \int |x^{p(\lambda+2)}| \alpha^{2p -1} \,|\phi'|^p \,|b|^p \, |\partial_x \log \sqrt h|^p x^{-n-1} \sqrt h dx dy \nn \\ 
&\lesssim & ||b||_p ^p \int_{\delta} ^{x_1} x^{2p-1} |\phi'|^p dx \; \lesssim \; ||b||_p ^p \int_\delta ^{2 \delta} x^{p-1} dx \nn \\ 
& \approx & ||b||_p ^p. \nn 
\eeq
Next we estimate 
\[ |VI|^p \lesssim \int_\delta ^{x_1} x^{2p-1} |\Delta_h b(y)|^p \sqrt h dx dy \; \lesssim \; \left \Vert \Delta_h b \right \Vert_p ^p, \]
with the norm on the right defined analogously to \eqref{eq:bpp}.  A completely analogous estimate shows that 
\[ |VII| \lesssim \left \Vert \sum_{i,j=1} ^n h^{ij} \partial_i (\log \alpha) \partial_j (b) \right \Vert_p.\]

Since $b$ is supported in a ball of radius $r = r_\epsilon$ in $Y$, we may choose $b$ to satisfy for a sufficiently large $N \in \N$, 
\[ || \Delta_h b||_p + \left \Vert \sum_{i,j=1} ^n h^{ij} \partial_i (\log \alpha) \partial_j (b) \right \Vert_p \lesssim r_\epsilon ^{-N} ||b||_p.\]
We therefore obtain that 
\[ II+III+IV+V+VI+VII \lesssim (1+r_\epsilon ^{-N}) ||b||_p.\]
Since it is quite likely, the way we have chosen the support of $b$, that $r_\epsilon \to 0$ as $\epsilon \to 0$, we may assume that $r_\epsilon < 1$ and therefore obtain 
\[ II + III+IV+V+VI+VII \lesssim r_\epsilon ^{-N} ||b||_p.\]
On the other hand, by \eqref{eq:Fpp} and \eqref{eq:relambda2}, we have the estimate 
\[ ||F||_p  \approx ||b||_p \log(1/\delta)^{1/p}. \]
To obtain the estimate 
\[ II+III+IV+V+VI+VII \lesssim \epsilon ||F||_p,\] 
it therefore is sufficient to require that 
\[ r_\epsilon^{-N} < \epsilon |\log(\delta)|^{1/p} \iff \delta < \exp \left( - r_\epsilon ^{-Np} \epsilon^{-p} \right). \] 

\begin{figure}[h!] 
\begin{center} 
\includegraphics[width=0.8\textwidth]{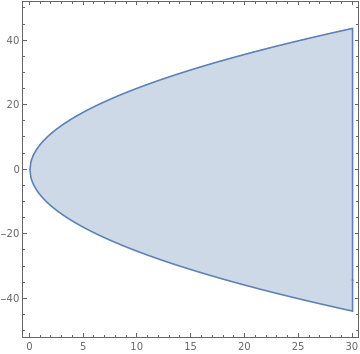}
\caption{The shaded region is contained in the $\cL^1$ spectrum.  In this example, we have taken $n=4$ (so the dimension of the manifold is $5$), and $\alpha_1=2$.} 
\label{fig:L1spec} 
\end{center} 
\end{figure} 

In this way we obtain that 
\[ || \Delta F(x,y) - \Lambda F(x,y)||_p \lesssim \epsilon ||F||_p.\]
The functions $F(x,y)$ are linearly independent because their supports are strictly increasing towards $\pa M$ as $\delta \to 0$, since $\delta \to 0$ as $\epsilon \to 0$.  We therefore obtain a family of approximate eigenfunctions for the spectral parameter  
\[ \Lambda = A \lambda (n-\lambda), \textrm{ for any } A \in \alpha^2 (Y), \] 
\[ \textrm{ for any $\lambda \in \C$ with } \Rea(\lambda) = \frac n p. \]
Consequently, all of these values of $\Lambda$ are contained in the spectrum of the Laplace operator acting on $\cL^p$ functions.  
Using interpolation as in the proof of \cite[Proposition 9]{Char}, we can show that whenever $1\leq p < q\leq 2$ the $\cL^q$ spectrum is contained in the $\cL^p$ spectrum. The only requirement for this containment is that the heat kernel of the Laplacian is unique, which is certainly true on conformally compact manifolds since their Ricci curvature is bounded below by a constant. 

So, writing $\lambda = \frac n q + is$, for $1 \leq q  \leq 2$, 
\[ A\left(\frac n q + is \right) \left( n- \frac n q - i s \right) = A\left( \frac n q \left(n-\frac n q \right) + s^2 + is\left( n - \frac{2n}{q} \right) \right).\]
The $\cL^p$ spectrum therefore contains the set 
\begin{equation} \label{eq:Lpspec_v1} \bigcup_{p \leq q \leq 2} \left\{ A \left( s^2 + \frac{n^2}{q} \left(1-1/q \right) \right) + i A s n \left( 1-\frac 2 q \right): s \in \R \right\} \subset \C \end{equation} 
for all $A \in \alpha^2(Y)$.  For $p=2$, we therefore obtain that the $\cL^2$ spectrum contains the ray 
\[ \left[ \frac{\alpha_0^2 n^2}{4}, \infty \right),\]
which was previously shown by Mazzeo \cite{mazzeo_88}.  
For $1 \leq p < 2$, this ray is also contained in the $\cL^p$ spectrum.  Next we aim to show that the set \eqref{eq:Lpspec_v1} is equal to 
\begin{equation} \left\{ x+iy: x \geq \frac{y^2}{An^2 (1-2/p)^2} + A \frac{n^2}{p} (1-1/p), \, y \in \R \right\} \subset \C. \label{eq:Aparabolas} \end{equation} 
If $s=0$ in \eqref{eq:Lpresv1}, then this set is 
\[ \left\{ A \frac{n^2}{q} \left( 1 - \frac 1 q \right) : p \leq q \leq 2 \right\}.\]
We observe that the function 
\[ \frac{y^2}{An^2(1-2/q)^2} + A \frac{n^2}{q} (1-1/q) \]
is a continuous and increasing function of $q \in [p, 2)$, tending to infinity as $q \to 2$.  So for a given 
\[ x \geq \frac{y^2}{An^2 (1-2/ p)^2} + A \frac{n^2}{p} (1-1/p)\]
there exists a unique $q \in [p,2)$ such that 
\[x= \frac{y^2}{An^2(1-2/q)^2} + A \frac{n^2}{q} (1-1/q) . \] 
Then, since $An(1-2/q) \neq 0$, we may set 
\[ s = \frac{y}{An(1-2/q)}  \implies As^2 = \frac{y^2}{A n^2 (1-2/q)^2}.\]
Writing both $x$ and $y$ in terms of $s$ we  therefore obtain that 
\[ x+iy = A s^2 + A \frac{n^2}{q} (1-1/q) + i A s n(1-2/q) \] 
is an element of the set \eqref{eq:Lpspec_v1}.  This shows a bijection between the sets \eqref{eq:Lpspec_v1} and \eqref{eq:Aparabolas}.

\begin{figure}[h!] 
\begin{center} 
\includegraphics[width=0.8\textwidth]{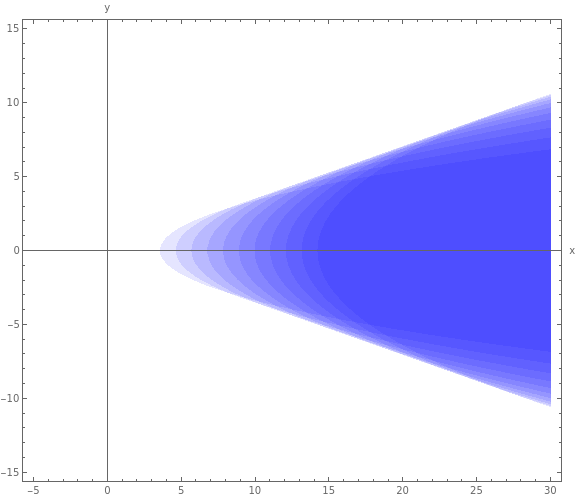}
\caption{The blue shaded region is the ensemble of parabolic regions that are contained in the $\cL^p$ spectrum.  In this example, we have taken $n=4$ (the dimension of the manifold is $5$), $p=\frac 3 2$, and $A \in [1,4]$ ($\alpha_0=1$ and $\alpha_1=2$).} 
\label{fig:Lpspec} 
\end{center} 
\end{figure}

For $p=1$, since $A \in \alpha(Y)$, and $\alpha_1^2 \in \alpha^2(Y)$ we obtain that the $\cL^1$ spectrum contains the parabolic region 
\beq \left\{ z = x+iy: x \geq \frac{y^2}{\alpha_1^2 n^2}, \quad y \in \R \right\}. \label{eq:Aparabola1} \eeq
The above parabola \eqref{eq:Aparabola1} gives the largest region we can obtain for the $\cL^1$ spectrum based on the values of $\alpha(Y)$. An example of this region is shown in Figure \ref{fig:L1spec}. Note that the vertex of the parabola is at $(0,0)$. This gives the first part of Theorem \ref{th1}.

For $p \in (1, 2)$, the $\cL^p$ spectrum contains the ensemble of parabolic regions 
\[ \left\{ z = x+iy: x \geq \frac{y^2}{A n^2 (1-2/p)^2} + A \frac{n^2}{p} \left( 1 - \frac 1 p \right), \, y \in \R, A \in \alpha^2(Y) \right\} \label{eq:Aparabolas2} \] 
which gives us the first part of Theorem \ref{th2}. 
As the parameter $A$ increases, the parabolas become wider and further shifted to the right.  Decreasing $A$ shifts the parabola to the left while making the parabola narrower.  An example of this ensemble of parabolc regions is shown in Figure \ref{fig:Lpspec}.  We note that $\alpha^2(Y)$ could be any of the following types of sets: 
\begin{enumerate} 
\item a single value if $|d\rho|$ is constant on $\pa M$; 
\item a finite discrete set of values if $\pa M$ consists of finitely many connected components on which $|d\rho|$ is constant  but not necessarily having the same value; 
\item the interval $[\alpha_0^2, \alpha_1^2]$ if $\pa M$ has one connected component on which $|d\rho|$ is non-constant and varies between $\alpha_0 < \alpha_1$; 
\item the union of finitely many closed intervals contained in $[\alpha_0^2, \alpha_1^2]$ if $\pa M$ has several connected components on which $|d\rho|$ is non-constant. 
\end{enumerate} 

It is interesting to consider what the enveloping curve to this family of parabolas is in case (3) above, when the range of the function $\alpha^2$ is exactly the interval  $[\alpha_0^2,\alpha_1^2]$. To find the enveloping curve we first differentiate the 1-parameter family of curves 
\[
F(A,y)=\frac{y^2}{A n^2 (1-2/p)^2} + A \frac{n^2}{p} \left( 1 - \frac 1 p \right)
\]
with respect to $A$.  We find the critical point $A=A(y)$. 
\begin{equation*}
    \begin{split}
        &\frac{\partial F(A,y)}{\partial A} = - \frac{1}{A^2} \, \frac{y^2 p^2}{n^2 (p-2)^2} + \frac{n^2 (p-1)}{p^2} =0\\
        \Longleftrightarrow & A= \pm \frac{y \, p^2}{n^2 |p-2| \, \sqrt{p-1}}.
    \end{split}
\end{equation*}
Substituting this critical value back into $x=F(A,y)$, yields the enveloping curve
\[
x= \pm \, y \; \frac{ 2 \,\sqrt{p-1} }{  |p-2|}  ,
\]
which are two lines passing through the origin. 

\begin{figure}[h!] 
\begin{center} 
\includegraphics[width=0.8\textwidth]{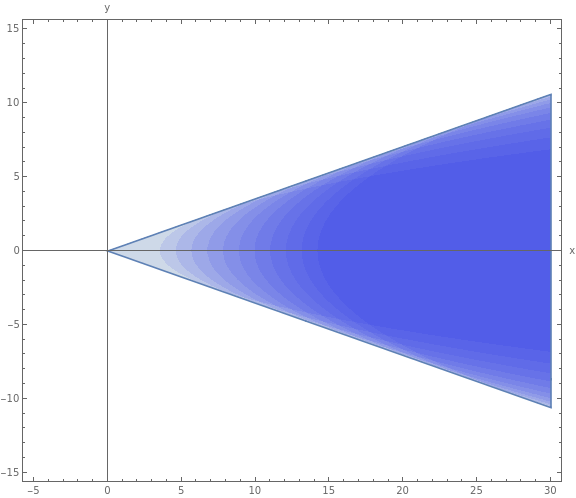}
\caption{This shows the region contained in the spectrum from Figure \ref{fig:Lpspec}, together with the enveloping curve.} 
\label{fig:envelope} 
\end{center} 
\end{figure} 

Each curve in the 1-parameter family $x=F(A,y)$ is tangent to the enveloping curve.  All of these curves are contained in the cone 
\[ \left\{ x+iy: x \geq  \, |y| \; \frac{ 2 \,\sqrt{p-1} }{  |p-2|} \right\}. \]   
This is depicted in Figure \ref{fig:envelope}.

\section{The resolvent set} \label{s:resolvent}
Here we show that the resolvent set contains the exterior of a certain parabolic region.  Consequently,  the spectrum of the Laplace operator is contained within a certain parabolic region that depends on $\alpha_1$, $n$, $p$, and the bottom of the $\cL^2(M)$ spectrum $\lambda_1$. To prove this, we will use a recent result due to the first author and Z. Lu which holds on manifolds with an exponential rate of volume growth defined as follows. 
\begin{Def} \label{def:volgrow} 
The exponential rate of volume growth of $M$, denoted by $\kappa$, is the infimum of all real numbers satisfying the property: for any $\eps>0$, there is a constant $C(\eps)$, depending only on $\eps$ and the dimension of the manifold, such that for any $p\in M$ and any $R\geq 1$, we have
\begin{equation}\label{3-1a}
{\rm Vol}(B_p(R))\leq C(\eps) {\rm Vol}(B_p(1))e^{(\kappa+\eps)R}.
\end{equation}
Above, $\Vol(B_p (R)$ is the volume of the ball of radius $R$ centered at $p$. 
We define $\kappa$ to be $\infty$ if for any $\kappa>0$ and  any $C>0$, we can find a pair $(p,R)$ such that 
\[ {\rm Vol}(B_p(R)) > C {\rm Vol}(B_p(1))e^{\kappa R}.\]
\end{Def}

With this notion of exponential rate of volume growth the first author and Z. Lu demonstrated the following result in  \cite{ChLu1}. 
\begin{thm}[NC \& Zhiqin Lu \cite{ChLu1}] \label{th:chlu1} 
Let $M$ be a complete manifold with Ricci curvature bounded below.  Denote by $\kappa$  the exponential rate of volume growth of $M$ as in Definition \ref{def:volgrow} and $\lambda_1$   the infimum  of the spectrum of the Laplacian on $\cL^2$. Let $z$ be a complex number such that ${|\rm Im}(z)|>\kappa/2$. Then
\[
(H-z^2)^{-1}
\]
is a bounded operator on $\cL^1(M)$, where
\[
H=\Delta-\lambda_1.
\]
Moreover, for $p \geq 1$, 
\[
(H-z^2)^{-1}
\]
is a bounded operator on $\cL^p(M)$, whenever  ${|\rm Im}(z)|>\left|\tfrac 1p - \tfrac 12\right|\kappa$.
\end{thm}

The authors actually proved the result for a more general class of operators, including the Laplacian on differential forms, with some additional assumptions on the curvature tensor in that case.  To apply Theorem \ref{th:chlu1} in our setting it suffices to demonstrate that our manifolds have a finite exponential rate of volume growth, which we prove below.
\begin{prop} \label{prop:vol_v2} Let $M$ be a conformally compact manifold of dimension $n+1$.  Then, the Ricci curvature of $M$ is bounded below, and the exponential rate of volume growth is at most $n \alpha_1$.  Here, 
\[ \alpha_1 := \max_{\pa M} |d\rho|.\] 
\end{prop} 
\begin{proof}
We follow the proof of \cite[Proposition 1]{sturm93}.  Fix $\varepsilon > 0$. By possibly increasing the size of the compact set $K \subset M$ in \eqref{eq:special_g}, assume that the sectional curvatures on $M \setminus K$ are bounded below by $-(\alpha_1 + \varepsilon)^2$.  This is possible because for points tending to $\pa M$, the sectional curvatures tend to $-\alpha^2 (y) \in [-\alpha_1^2, -\alpha_0^2]$.  Since the dimension of $M$ is $n+1$, we therefore have that Ricci curvature is bounded below by $-n(\alpha_1+\varepsilon)^2$ on $M \setminus K$.  

Let $R$ be the diameter of $K$.  For a point $p_o \in M$, let $s=s_\varepsilon(p_o) = d(p_o, K)$, and $t=t_\varepsilon (p_o) = s+R$.  As observed in \cite{sturm93}, $t-s$ does not depend on $p_o \in M$.  Since the Ricci curvature is uniformly bounded from below on $M\setminus K$, it is uniformly bounded from below on all of $M$.  So, there is some $\gamma > 0$ such that the Ricci curvature on $M$ is bounded from below by $-n \gamma^2$.  

As in \cite{sturm93}, we introduce the Sturm-Liouville equation on $\R$, 
\beq \label{eq:slp1} u'' + q u = 0, \, u(0) = 0,\, u'(0) = 1, \eeq 
with 
\[ q(r) = \begin{cases} - (\alpha_1+\varepsilon)^2, & \textrm{ for } r \in [0, s), \\ - \gamma^2, & \textrm{ for } r \in [s, t), \\ - (\alpha_1 + \varepsilon)^2, & \textrm{ for } r \in [t, \infty).  \end{cases} \]
By following the same proof as in \cite[p. 450--452]{sturm93}, which relies on the properties of the Sturm Liouville equation as in Bishop's comparison theorem, we get the upper volume estimate 
\[ \Vol(B_{p_o} (r)) \leq C(\varepsilon) \Vol(B_{p_o} (1)) e^{(\alpha_1+\varepsilon)r}\]
for a uniform constant $C(\varepsilon)$ which is independent of $p_o$. Since we are able to obtain such an estimate for any $\varepsilon > 0$, it follows that the exponential rate of volume growth of $M$ is less than or equal to $n \alpha_1$. 
\end{proof}

As a consequence, the exponential rate of volume growth of a conformally compact manifold is at most $n \alpha_1$, so in particular, it is finite.

\begin{cor} \label{cor:taylor} Let $M$ be a conformally compact manifold of dimension $n+1$.  Let $\alpha_0$ denote the minimum value of $|d\rho|$ over the boundary of $M$ and $\lambda_1$ the bottom of the $\cL^2(M)$ spectrum of the Laplacian.  Let $\kappa$ denote the exponential rate of volume growth of $M$ as in Definition \ref{def:volgrow}. Then the region
\beq
\left\{ \lambda_1  +z^2  : \, |{\rm Im}(z)|>\frac{\kappa}{2} \right\}  \subset \C \label{eq:L1resv1}
\eeq
is contained in the resolvent set of the Laplace operator acting on $\cL^1(M)$.  Moreover, 
\beq \left \{ \lambda_1 + z^2 : |\Ima(z)| > \left| \frac 1 p - \frac 1 2 \right| \kappa \right\}  \subset \C \label{eq:Lpresv1} \eeq 
is contained in the resolvent set of the Laplace operator acting on $\cL^p(M)$ for $ p\geq 1$. In general $0<\lambda_1 \leq \frac{n^2 \alpha_0^2}{4}$, and whenever the Laplacian on $\cL^2(M)$ has no  isolated eigenvalues, then  $\lambda_1=\frac{n^2 \alpha_0^2}{4}$.
\end{cor} 
\begin{proof} 
By Proposition \ref{prop:vol_v2}, $M$ satisfies the hypotheses of Theorem \ref{th:chlu1}. By \cite{mazzeo_88}, the infimum of the essential spectrum of the Laplacian on $\cL^2(M)$ is $\frac{n^2 \alpha_0^2}{4}$. Therefore the spectrum of the Laplacian on $\cL^2$ is $[\frac{n^2 \alpha_0^2}{4},\infty)$ possibly with a finite set of isolated eigenvalues $\{\lambda_1, \ldots, \lambda_m\}$ with $0 < \lambda_m < \frac{n^2 \alpha_0^2}{4}$.   The corollary now follows.
\end{proof}

With the corollary above, we shall finish the proof of Theorems \ref{th1} and \ref{th2} and establish the precise value of the volume growth in Corollary \ref{cor2}. 

\begin{proof}[Completion of proofs of Theorems \ref{th1} and \ref{th2} and Corollary \ref{cor2}]

Let $z=t+is$.  Since $z^2=(-z)^2$, the set \eqref{eq:L1resv1} may also be described as 
\[ \left\{  \lambda_1  + t^2 - s^2 + 2its : s^2 > \frac {\kappa^2}{4}, \, t \in \R \right\}.\]
Let us call $y=2ts$. Replacing $t$,  the set \eqref{eq:L1resv1} may also be described
\[ \left\{ \lambda_1 + \frac{y^2}{4s^2} - s^2 + i y: y \in \R, \, s^2 > \frac {\kappa^2} 4 \right\}\]
\beq = \left\{ x + i y : x < \lambda_1  + \frac{y^2}{\kappa^2} - \frac{\kappa^2}{4}, \, y \in \R \right\}. \label{eq:l1_res} \eeq

Consequently, since the set \eqref{eq:l1_res} is contained within the resolvent set, the $\cL^1$ spectrum is contained within the complement of that region and is therefore contained within the parabolic region 
\beq  \left\{ x + i y : x \geq \lambda_1 + \frac{y^2}{\kappa^2} - \frac{\kappa^2}{4} , \, y \in \R \right\}. \label{eq:Aparabola1_rv2} \eeq

Combining our results for the $\cL^1$ spectrum allows us to precisely specify the exponential rate of volume growth of a conformally compact manifold as stated in Corollary \ref{cor2}.  The parabolic region 
\begin{equation} \label{eq:pfcor1} \left\{ z = x+iy: x \geq \frac{y^2}{\alpha_1^2 n^2}, \, y \in \R \right\} \end{equation} 
    is contained in the $\cL^1$ spectrum. By Proposition \ref{prop:vol_v2}, $\kappa \leq n \alpha_1$. If $\kappa<\alpha_1 n$ this gives rise to a contradiction because the parabolic region \eqref{eq:Aparabola1_rv2} that \em contains \em the $\cL^1$ spectrum is strictly \em narrower \em than the parabolic region \eqref{eq:pfcor1} contained within the $\cL^1$ spectrum.  Consequently, the exponential volume growth parameter $\kappa \geq \alpha_1 n$.  By Proposition \ref{prop:vol_v2}, we therefore obtain the equality $\kappa = n \alpha_1$. 
An example of the parabolic region that contains the $\cL^1$ spectrum together with the parabolic region that is contained within the $\cL^1$ spectrum is  depicted in \ref{fig:L1spec_res}.

\begin{figure}[h!] 
\begin{center} 
\includegraphics[width=0.6\textwidth]{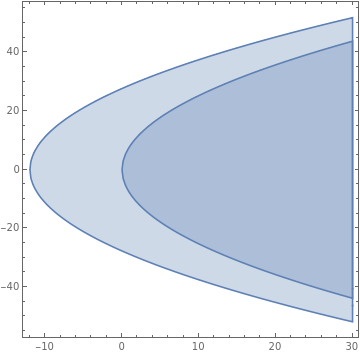}
\caption{The larger blue shaded region contains the $\cL^1$ spectrum, and the smaller shaded region is contained in the $\cL^1$ spectrum.  In this example, the dimension is 5 ($n=4$), $\alpha_0=1$, $\alpha_1=2$, and there are no isolated $\cL^2$ eigenvalues. }
\label{fig:L1spec_res} 
\end{center} 
\end{figure}

Next, assume that $1<p \leq 2$.  Let $z=t+is$, and note that that $\alpha_1>0$ so that $s > 0$.  Then setting $\kappa = n \alpha_1$ the set \eqref{eq:Lpresv1} is 

\[ \left\{\lambda_1  + t^2 - s^2 +2its : t \in \R, \, s > \frac {n \alpha_1} 2 \left( \frac 2 p - 1 \right) \right\}.\]
Setting $y=2ts$ and observing that $s>0$, this set is 
\beq \label{eq:lp_res} & \displaystyle{ \left \{ \lambda_1 + \frac{y^2}{4s^2} - s^2 + i y: y \in \R, s > \frac {n \alpha_1} 2 \left( \frac 2 p - 1 \right) \right\}  }\\ 
&= \displaystyle{ \left \{ x + i y : x < \lambda_1 + \frac{y^2}{n^2 \alpha_1 ^2 \left( \frac 2 p - 1\right)^2} - \frac{ \left( \frac 2 p - 1 \right)^2 n^2 \alpha_1^2}{4} , y \in \R \right\}. } \nn \eeq 
The $\cL^p$ spectrum is therefore contained in the complement of \eqref{eq:lp_res}, which is the parabolic region given in Theorem \ref{th2}.  An example of the parabolic region that contains the $\cL^p$ spectrum together with the one-parameter family of parabolic regions that are contained within the $\cL^p$ spectrum are depicted in Figure \ref{fig:Lpspec_res}. Note that for $p\geq 2$, the parabolas given by the dual dimension  $q$ corresponding to $1/q=1-1/p$ are exactly the same, hence the result extends to all $p \geq 1$.

\begin{figure}[h!] 
\begin{center} 
\includegraphics[width=0.8\textwidth]{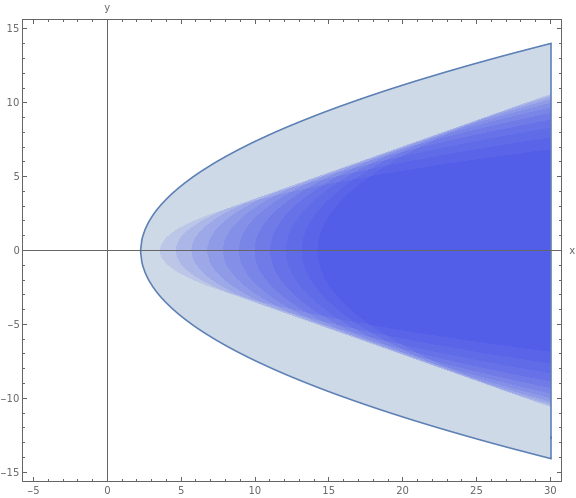}
\caption{The larger blue shaded region conatins the $\cL^p$ spectrum, and the smaller shaded region is contained in the $\cL^p$ spectrum.  In this example, $p=3/2$, the dimension is 5 ($n=4$), $\alpha_0=1$, and $\alpha_1=2$.} 
\label{fig:Lpspec_res} 
\end{center} 
\end{figure} 

\end{proof}




\begin{remark}
We would like to remark that we could have also found the parabola corresponding to the resolvent set using a similar theorem due to Taylor in \cite{taylor89}, instead of Theorem \ref{th:chlu1}.  Taylor considered manifolds with injectivity radius uniformly bounded below  and with bounded geometry, which satisfy the volume estimate
\begin{equation}  \label{eq:taylor_vol}{\rm Vol}(B_x(R)) \leq  C(1+R^2)^{1/(2\mu)} \,e^{\kappa R}, \quad \forall R>0, 
\end{equation} 
where is $\mu$ and $\kappa$ are non-negative constants, and $C$ is uniform.  Under the above assumptions and for $p \geq 1$, he showed that $(H-z^2)^{-1}$ is a bounded operator on $\cL^p(M)$, whenever  $|{\rm Im}(z)|>\left|\tfrac 1p - \tfrac 12\right|\kappa$. Here $H=\Delta -\lambda_1$ as in our case. 

Since conformally compact manifolds have Ricci curvature bounded below, the volume of balls of radius 1 is uniformly bounded above, hence  \eqref{eq:taylor_vol} immediately follows from Proposition \ref{prop:vol_v2}. To obtain the strict positive lower bound on the injectivity radius, it suffices to show that the volume of balls of radius $1$ has a uniform positive lower bound. This is true given the structure of our manifolds at infinity and the fact that they have infinite volume. It then follows by \cite[Theorem 4.7]{CGT} that the injectivity radius is uniformly bounded below by a positive constant.  Since Theorem \ref{th:chlu1} is a somewhat more general result, and with the intent to expand this work to a larger class of manifolds and operators, we chose to use it instead.
\end{remark}


\section{ Eigenvalues} \label{s:discrete} 
In this section we investigate the set of eigenvalues in the $\cL^p$ spectrum.  We show that for all $p>2$, the eigenvalues of the Laplace operator acting on $\cL^2$ are also eigenvalues of the Laplace operator acting on $\cL^p$.  Moreover, we show that every eigenvalue of the Laplace operator acting on $\cL^p$ for $1 \leq p < 2$ is an $\cL^2$ eigenvalue as well.  
To obtain these results we use the following fact about the heat operator.

\begin{lem} \label{le:heat}
The heat operator over a conformally compact manifold is bounded from $\cL^p(M)$ to $\cL^q(M)$ for any $1\leq p \leq q \leq \infty.$
\end{lem}
\begin{proof}
Since the Ricci curvature of a conformally compact manifold is uniformly bounded below, the estimates of Saloff-Coste \cite{SaCo}, give that the heat kernel $h_t(x,y)$ of the Laplacian on functions is bounded above by  
\[
h_t(x,y) \leq C  \,\Vol(B_x (\sqrt t))^{-1/2}\Vol(B_y (\sqrt t))^{-1/2} e^{\sqrt{K t}}   \; e^{-\frac{d^2(x,y)}{C' t}}. 
\]
Here $C$ and  $C'$ only depend on $n$, while $K$ depends on the Ricci curvature lower bound. Moreover, by Bishop's volume comparison theorem there exists a uniform constant $C$ independent of $x$ and  $r>0$ such that
\[
\frac{\Vol(B_x (1))}{\Vol(B_x (r))} \leq C \max\{1, r^{-(n+1)} \} .
\]
Since the the volume of balls of radius 1 is uniformly bounded below, as we have previously discussed, it follows that 
\[
h_t(x,y) \leq C  \max\{1, t^{-(n+1)} e^{\sqrt{K t}} \} =a_t. 
\]
Here $C$ only depends on $K$ and $n$. By \cite[Lemma 2.1.2]{davies89} we conclude that the heat operator is ultracontractive, it is in other words bounded from $\cL^2(M)$ to $\cL^\infty(M)$ with norm $c_t=a_t^{1/2}$. By taking adjoints, this also implies that the heat operator is bounded from $\cL^1$ to $\cL^2$, and by using repeated interpolation we get that it is bounded from $\cL^p$ to $\cL^q$ for any $1\leq p \leq q \leq \infty,$ since it is always bounded on $\cL^p$.  
\end{proof}


\begin{proof}[Proof of Theorem \ref{th:discrete}]
Suppose that $\lambda_i$ is an eigenvalue of the Laplacian on $\cL^2(M)$ with corresponding eigenfunction $\phi_i \in \cL^2(M)$. The heat operator preserves eigenfunctions, in other words 
\[
e^{-t\Delta} \phi_i = e^{-t\lambda_i} \phi_i
\]
for all $t$. Since the heat operator is ultracontractive we have 
\[
\|e^{-t\lambda_i} \phi_i \|_{\infty} =\|e^{-t\Delta} \phi_i\|_{\infty} \leq c_t \|\phi_i\|_2
\]
for all $t>0$.  Taking $t=1/\lambda_i$ if $\lambda_i>0$, and $t=1$ otherwise, we get
\[
\|\phi_i \|_{\infty} \leq C \|\phi_i\|_2.
\]
In other words, any $\cL^2$ eigenfunction must also be bounded, and hence by interpolation it must also belong to $\cL^p$ for any $2\leq p \leq \infty$. In other words, any eigenvalue in the $\cL^2$-spectrum, including the discrete isolated eigenvalues, must also be an eigenvalue for the $\cL^p$ spectrum for all $p>2$.

Similarly, if $\lambda_i$ is an $\cL^p$-eigenvalue of the Laplacian for some $1\leq p<2,$ then it must also be an $\cL^2$-eigenvalue. As a result, the $\cL^p$ eigenvalues for $1\leq p<2$ can only be points on the real line, and contained in the $\cL^2$ spectrum.
\end{proof}

Theorem \ref{th:discrete} reflects the absence of duality between the $\cL^p$  and the $\cL^q$ spectra for $\frac 1q =1 -\frac 1p$ also observed by \cite{DST,JW1,JW}. We further observe that the analytic results for the $\cL^p$ spectrum do not allow one to conclude that if a point in the complex plane is an isolated point of the $\cL^p$ spectrum for $p\neq2$ then it must be an eigenvalue.

\section{Concluding Remarks} \label{s:conclude}
We have shown in Theorem \ref{th1} that the Laplacian on a conformally compact manifold acting on $\cL^1(M)$ contains a parabolic region of the complex plane and is also contained in a parabolic region of the complex plane.  This is depicted in Figure \ref{fig:L1spec_res}.  If the manifold is also asymptotically hyperbolic, and there are no isolated $\cL^2$ eigenvalues, then our result is sharp; we have precisely calculated the $\cL^1$ spectrum.  However, if the sectional curvatures are variable up to the boundary, then the two parabolic regions are not identical.  This may reflect the difficulties caused by variable curvature at infinity Borthwick faced when studying the resolvent kernel acting on $\cL^2$ \cite{borthwick}.  It is not immediately clear whether we may be able to completely determine the $\cL^p$ spectrum in the case of variable curvature, but this will be interesting to study further.  Moreover, our results show that for a very large class of manifolds that include many interesting examples, the $\cL^p$ spectrum \em depends on $p$.  \em  Since different $\cL^p$ spaces may be relevant to certain specific physical processes like diffusion, this shows that the value of $p$ may crucially affect such processes on conformally compact manifolds.  In particular, when the conformally compact manifold is used as a model in general relativity, this phenomenon is important to bear in mind.  Although it may be known by experts, it is interesting that we were able to combine geometric techniques with results for the $\cL^1$ spectrum of the Laplace operator in order to obtain the precise exponential rate of volume growth of any conformally compact manifold.  Perhaps in other contexts, one may similarly be able to obtain information about volume growth estimates through the study of the Laplace spectrum.

\section*{Acknowledgements}
We are grateful to the anonymous referees for constructive critiques that improved the quality of the manuscript and to Zhiqin Lu for keen observations. JR is grateful to Ksenia Fedosova for help producing the figures and to Eric Bahuaud, Klaus Kr\"oncke, and Raquel Perales for fruitful discussions. NC was partially supported by a University of Cyprus Internal Grant and JR was supported by the Swedish Research Council grant 2018-03873 while this work was in progress.

\end{document}